\documentclass[letterpaper, 10 pt, conference]{ieeeconf}
\usepackage{amsmath}
\usepackage{amsthm}
\usepackage{amssymb}
\usepackage{graphicx}
\usepackage{url}
\usepackage{bbm}
\usepackage{epsfig}

\IEEEoverridecommandlockouts
\begin{document}
\newtheorem{rem}{Remark}

\newtheorem{lem}{Lemma}
\newtheorem{thm}{Theorem}
\newtheorem{example}{Example}
\newtheorem{assumption}{Assumption}
\newtheorem{prop}{Proposition}
\newtheorem{defn}{Definition}

\title{\LARGE \bf
Extracting Flexibility of Heterogeneous Deferrable Loads via Polytopic Projection Approximation
}

\author{Lin Zhao, He Hao, and Wei Zhang
\thanks{
L. Zhao and W. Zhang are with the Dept. of Electrical and Computer Engineering, The Ohio State University, Columbus, OH, USA, 43210
{\small (email: zhao.833@osu.edu; zhang.491@osu.edu)}
\newline 
\indent H. Hao is with Pacific Northwest National Laboratory, P.O. Box 999, 99352, Richland, Washington, USA
{\small (email: He.Hao@pnnl.gov)}}
}

\maketitle
\begin{abstract}
Aggregation of a large number of responsive loads presents great power
flexibility for demand response. An effective control and coordination
scheme of flexible loads requires an accurate and tractable model
that captures their aggregate flexibility. This paper proposes a novel
approach to extract the aggregate flexibility of deferrable loads
with heterogeneous parameters using polytopic projection approximation.
First, an exact characterization of their aggregate flexibility is
derived analytically, which in general contains exponentially many
inequality constraints with respect to the number of loads. In order
to have a tractable solution, we develop a numerical algorithm that
gives a sufficient approximation of the exact aggregate flexibility.
Geometrically, the flexibility of each individual load is a polytope,
and their aggregation is the Minkowski sum of these polytopes. Our
method originates from an alternative interpretation of the Minkowski
sum as projection. The aggregate flexibility can be viewed as the
projection of a high-dimensional polytope onto the subspace representing
the aggregate power. We formulate a robust optimization problem to
optimally approximate the polytopic projection with respect to the
homothet of a given polytope. To enable efficient and parallel computation
of the aggregate flexibility for a large number of loads, a muti-stage
aggregation strategy is proposed. The scheduling policy for individual
loads is also derived. Finally, an energy arbitrage problem is solved
to demonstrate the effectiveness of the proposed method.
\end{abstract}

\section{Introduction}

The future power system will be modernized with advanced metering
infrastructure, bilateral information communication network, and intelligent
monitoring and control system to enable a smarter operation~\cite{BhattShahJani2014}.
The transformation to the smart grid is expected to facilitate the
deep integration of renewable energy, improve the reliability and
stability of the power transmission and distribution system, and increase
the efficiency of power generation and energy consumption.

Demand response program is a core subsystem of the smart grid, which
can be employed as a resource option for system operators and planners
to balance the power supply and demand. The demand side control of
responsive loads has attracted considerable attention in recent years~\cite{Callaway2009,Koch2011,Kamgarpour2013a,ZhaoZhang2015}.
An intelligent load control scheme should deliver a reliable power
resource to the grid, while maintaining a satisfactory level of power
usage to the end-user. One of the greatest technical challenges of
engaging responsive loads to provide grid services is to develop control
schemes that can balance the aforementioned two objectives~\cite{CallawayPIEEE}.
To achieve such an objective, a hierarchical load control structure
via aggregators is suggested to better integrate the demand-side resources
into the power system operation and control~\cite{RuizCobeloOyarzabal2009,CallawayPIEEE}.

In the hierarchical scheme, the aggregator performs as an interface
between the loads and the system operator. It aggregates the flexibility
of responsive loads and offers it to the system operator. In the meantime,
it receives dispatch signals from the system operator, and execute
appropriate control to the loads to track the dispatch signal. Therefore,
an aggregate flexibility model is fundamentally important to the design
of a reliable and effective demand response program. It should be
detailed enough to capture the individual constraints while simple
enough to facilitate control and optimization tasks. Among various
modeling options for the adjustable loads such as thermostatically
controlled loads (TCLs), the average thermal battery model~\cite{Hao2015,ZhaoZhang2016,Mathieu2013a,Kamgarpour2013a}
aims to quantify the aggregate flexibility, which is the set of the
aggregate power profiles that are admissible to the load group. It
offers a simple and compact model to the system operator for the provision
of various ancillary services. Apart from the adjustable loads, deferrable
loads such as pools and plug-in vehicles (PEVs) can also provide significant
power flexibility by shifting their power demands to different time
periods. However, different from the adjustable loads, it is more
difficult to characterize the flexibility of deferrable loads due
to the heterogeneity in their time constraints.

In this paper, we focus on modeling the aggregate flexibility for
control and planning of a large number of deferrable loads. There
is an ongoing effort on the characterization of the aggregate flexibility
of deferrable loads~\cite{NayyarTaylorSubramanianEtAl2013,LiuLiZhangEtAl2013,Hao2014,Chen2015}.
An empirical model based on the statistics of the simulation results
was proposed in~\cite{LiuLiZhangEtAl2013}. A necessary characterization
was obtained in~\cite{NayyarTaylorSubramanianEtAl2013} and further
improved in~\cite{Hao2014}. For a group of deferrable loads with
homogeneous power, arrival time, and departure time, a majorization
type exact characterization was reported in~\cite{Hao2014}. With
heterogeneous departure times and energy requirements, a tractable
sufficient and necessary condition was obtained in~\cite{Chen2015},
and was further utilized to implement the associated energy service
market~\cite{ChenQiuVaraiya2015}. Despite these efforts, a sufficient
characterization of the aggregate flexibility for general heterogeneous
deferrable loads remains a challenge.

To address this issue, we propose a novel geometric approach to extract
the aggregate flexibility of heterogeneous deferrable loads. Geometrically,
the aggregate flexibility modeling amounts to computing the Minkowski
sum of multiple polytopes, of which each polytope represents the flexibility
of individual load. However, calculating the Minkowski sum of polytopes
under facet representation is generally NP-hard~\cite{Tiwary2008}.
Interestingly, we are able to show that for a group of loads with
general heterogeneity, the exact aggregate flexibility can be characterized
analytically. But the problem remains in the sense that there are
generally exponentially many inequalities with respect to the number
of loads and the length of the time horizons, which can be intractable
when the load population size or the number of steps in the considered
time horizon is large. Therefore, a tractable characterization of
the aggregate flexibility is desired.

For deferrable loads with heterogeneous arrival and departure times,
the constraint sets are polytopes that are contained in different
subspaces. Alternative to the original definition of the Minkowski
sum, we find it beneficial to regard it as a projection operation.
From the latter perspective, the aggregate flexibility is considered
as the projection of a higher dimensional polytope to the subspace
representing the aggregate power of the deferrable loads. Therefore,
instead of approximating the Minkowski sum directly by its definition,
we turn to approximating the associated projection operation. To this
end, we formulate an optimization problem which approximates the projection
of a full dimensional polytope via finding the maximum homothet of
a given polytope, i.e., the dilation and translate of that polytope.
The optimization problem can be solved very efficiently by solving
an equivalent linear program. Furthermore, we propose a ``divide
and conquer'' strategy which enables efficient and parallel computation
of the aggregate flexibility of the load group. The scheduling policy
for each individual load is derived simultaneously along the aggregation
process. Finally, we apply our model to the PEV energy arbitrage problem,
where given predicted day-ahead energy prices, the optimal power profile
consumed by the load group is calculated to minimize the total energy
cost. The simulation results demonstrate that our approach is very
effective at characterizing the feasible aggregate power flexibility
set, and facilitating finding the optimal power profile.

There are several closely-related literature on characterizing flexibility
of flexible loads. In our previous work~\cite{ZhaoZhang2016}, a
geometric approach was proposed to optimally extract the aggregate
flexibility of heterogeneous TCLs based on the given individual physical
models. The simulation demonstrated accurate characterization of the
aggregate flexibility which was very close to the exact one. However,
this approach cannot be applied to the deferrable loads directly.
Similar to~\cite{ZhaoZhang2016} which sought a special class of
polytopes to facilitate fast calculation of Minkowski sum, the authors
in~\cite{MullerSundstromSzaboEtAl2015} proposed to characterize
the power flexibility using Zonotopes. Different from~\cite{ZhaoZhang2016},
this method could deal with the time heterogeneity as appeared in
the deferrable loads. In addition, both approaches extracted the flexibility
of \textit{individual} load.\textbf{ }In comparison, the approach
proposed in this paper is a \textit{batch} processing method: it directly
approximates the aggregate flexibility of a group of loads, which
could mitigate the losses caused by the individual approximation as
emphasized in~\cite{MullerSundstromSzaboEtAl2015}.

\textbf{Notation}: The facet representation of a polytope $\mathcal{P}\subset\mathbb{R}^{m}$
is a bounded solution set of a system of finite linear inequalities
\cite{Henk1997}: $\mathcal{P}:=\{x:Ax\leq c\}$, where throughout
this paper $\leq$ (or $<$, $\geq$,~$>$) means elementwise inequality.
A polytope $\mathcal{P}\subset\mathbb{R}^{m}$ is called full dimensional,
if it contains an interior point in $\mathbb{R}^{m}$. Given a full
dimensional polytope $\mathcal{P}$ in $\mathbb{R}^{m}$, a scale
factor $\lambda>0$, and a translate factor $\mu\in\mathbb{R}^{m}$,
the set $\lambda\mathbb{B}+\mu$ is called a homothet of $\mathbb{B}$.
We use $\biguplus$ to denote the Minkowski sum of multiple sets,
and $\oplus$ of two sets. We use ${\bf 1}_{m}$ to represent the
$m$ dimensional column vector of all ones, $I_{m}$ the $m$ dimensional
identity matrix, and $\mathbbm{1}_{C}$ the indicator function of
the set $C$. The bold $\mathbf{0}$ denotes the column vector of
all $0$'s with appropriate dimension. For two column vectors $u$
and $v$, we write $(u,v)$ for the column vector $[u^{T},v^{T}]^{T}$
where no confusion shall arise.

\section{Problem Formulation}

We consider the problem of charging a group of $N$ PEVs. The energy
state of each PEV can be described by a discrete time difference equation
on a finite time horizon $[0,T\delta]$, 
\begin{equation}
x_{t}=x_{t-1}+\delta u_{t},\label{eq:deferrable}
\end{equation}
where $x_{t}$ is the state of charge (SoC) with initial condition
$x_{0}=0$, and $u_{t}\in[0,p]$ with $p>0$ is the charging power
supplied to the vehicle during $[(t-1)\delta,t\delta)$. Let $t\in\mathrm{\mathbb{T}}:=\{1,2,\dots m\}$
denote the time interval $[(t-1)\delta,t\delta)$, where without loss
of generality, we will assume the time unit $\delta=1$ hour in the
sequel. Moreover, the PEV must be charged during a time window $\mathbb{A}:=\{a,\dots,d\}\subset\mathrm{\mathbb{T}}$
where $a$ is its arrival time, $d$ is its departure time, and $a<d$..
At the deadline $d$, the PEV is supposed to be charged with an SoC
$x_{d}\in[\underline{E},\bar{E}]$, where we assume that $\bar{E}^{i}>\underline{E}^{i}$.
The load is called deferrable if $\bar{E}<(d-a)p$. A charging power
profile $u:=[u_{1},\dots,u_{m}]^{T}$ is called admissible if the
load is only charged within $\mathbb{\mathbb{A}}$ and its SoC at
$t=m$ is within $[\underline{E},\bar{E}]$.

We differentiate the $i^{\mbox{th}}$ PEV ($\forall i\in\mathcal{N}:=\{1,2,\cdots,N\}$)
by using a superscript $i$ on the variables introduced above. The
charging task of the $i^{\mbox{th}}$ PEV is determined by $\Omega^{i}:=\{a^{i},d^{i},p^{i},\underline{E}^{i},\bar{E}^{i}\}$.
Let $\mathcal{P}^{i}$ be the set of all the admissible power profiles
of the $i^{\mbox{th}}$ load. It can be described as,
\begin{equation}
\mathcal{P}^{i}:=\left\{ \begin{array}{l}
u^{i}\in\mathbb{R}^{m}:u_{t}^{i}\in[0,p^{i}],\forall t\in\mathbb{A}^{i},\\
u_{t}^{i}=0,\forall i\in\mathrm{\mathbb{T}}\backslash\mathbb{\mathbb{A}}^{i}\mbox{ and }{\bf 1}_{m}^{T}u^{i}\in[\underline{E}^{i},\bar{E}^{i}]
\end{array}\right\} \label{eq:pi}
\end{equation}

It is straightforward to see that each $\mathcal{P}^{i}$ is a convex
polytope. In addition, we say $\mathcal{P}^{i}$ is of codimension
$m-\left|\mathbb{A}^{i}\right|$ if its affine hull is a $\left|\mathbb{A}^{i}\right|$
dimensional subspace of $\mathbb{R}^{m}$.

In the smart grid, the aggregator is responsible for procuring a generation
profile from the whole market to service a group of loads. We define
the generation profile that meets the charging requirements of all
PEVs as follows.
\begin{defn}
A generation profile $u$ is called \textit{adequate} if there exists
a decomposition $u=\sum_{i=1}^{N}u^{i}$, such that $u^{i}$ is an
admissible power profile for the $i^{\mbox{th}}$ load, i.e., $u^{i}\in\mathcal{P}^{i}$. 
\end{defn}
We call the set of all the adequate generation profiles the \textit{aggregate
flexibility} of the load group. It can be defined as the Minkowski
sum of the admissible power sets of each load,
\begin{multline}
\mathcal{P}=\biguplus_{i=1}^{N}\mathcal{P}^{i}:=\left\{ u\in\mathbb{R}^{m}:u=\sum_{i=1}^{N}u^{i},~u^{i}\in\mathcal{P}^{i}\right\} .\label{eq:P_AF}
\end{multline}

It is straightforward to show that $\mathcal{P}$ is also a convex
polytope whose codimension is to be determined by the parameters $\Omega$
of the deferrable loads.

\section{Exact Characterization of the Aggregate Flexibility}

The numerical complexity of the existing algorithms for calculating
the Minkowski sum is rather expensive (See~\cite{Weibel2007} for
some numerical results). In general, calculating $\mathcal{P}_{1}\oplus\mathcal{P}_{2}$
when $\mathcal{P}_{1}$ and $\mathcal{P}_{2}$ are polytopes specified
by their facets is NP-hard~\cite{Tiwary2008}. However, for the particular
problem of PEV charging, it is possible to characterize the exact
aggregate flexibility $\mathcal{P}$ analytically. Such characterization
is built on the results from the matrix feasibility problem and from
the network flow theory, both of which are intrinsically connected
with the PEV charging problem.
\begin{thm}
\label{thm:iff}Consider a group of PEVs or deferrable loads with
heterogeneous parameters $\Omega^{i}=\{a^{i},d^{i},p^{i},\underline{E}^{i},\bar{E}^{i}\}$,
$i\in\mathcal{N}$. Then the set $\mathcal{P}$ of adequate generation
profiles consists of those $u=[u_{1},u_{2},\cdots,u_{m}]^{T}$ which
satisfy 
\begin{multline}
\min\left\{ \sum_{i\in\alpha}\bar{E}^{i}-\sum_{t\in\beta}u_{t},\ \sum_{t\in\mathrm{\beta^{c}}}u_{t}-\sum_{i\in\alpha^{c}}\underline{E}^{i}\right\} \\
\geq-\sum_{i\in\alpha^{c}}\left|\beta\cap\mathbb{A}^{i}\right|p^{i},\label{eq:FullHetero}
\end{multline}
for all subsets $\alpha\subset\mathcal{N}$ and $\beta\subset\mathrm{\mathbb{T}}$,
where $\alpha^{c}$ and $\beta^{c}$ are the complement sets of $\alpha$
and $\beta$ in $\mathcal{N}$ and $\mathrm{\mathbb{T}}$, respectively.
\end{thm}
\begin{proof}
We first interpret the characterization of $\mathcal{P}$ as a matrix
feasibility problem. By definition, if a generation profile $u$ is
adequate, then there exists a decomposition $u=\sum_{i=1}^{N}u^{i}$
such that $u^{i}$ completes the $i^{\mbox{th}}$ PEV's charging task.
This is equivalent to the existence of a $N\times m$ matrix $M$,
the $i^{\mbox{th}}$ row of which is an admissible power profile of
the $i^{\mbox{th}}$ PEV. The matrix $M$ will be referred to as the
charging matrix. Given $u\in\mathcal{P}$, let $\mathcal{M}(u)$ denote
the set of all such charging matrices. These matrices have special
structures: the columns indexed by $\mathbb{A}^{i}$ in the $i^{\mbox{th}}$
row are the free positions which can be filled with a real number
in $[0,p^{i}]$, while the rest of the positions in this row are forbidden
positions that can only be filled with $0$'s. Moreover, $\forall M\in\mathcal{M}(u)$,
it has the $t^{\mbox{th}}$ column sum $u_{t}$ and the $i^{\mbox{th}}$
row sum in the interval $[\underline{E}^{i},\bar{E}^{i}]$. Clearly,
the non-emptiness of the set gives the condition for being an adequate
generation profiles, i.e, 
\[
\mathcal{P}=\{u\in\mathbb{R}^{m}:\ \mathcal{M}(u)\neq\emptyset\}.
\]

Furthermore, the condition for $\mathcal{M}(u)\neq\emptyset$ can
be derived by applying~\cite[Theorem 2.7]{Hershkowitz1997} to the
matrix case. By the definition in~\cite{Hershkowitz1997}, partitions
$\{r_{1},r_{2},\cdots,r_{N}\}$ and $\{c_{1},c_{2},\cdots,c_{m}\}$
of a sequence are said to be orthogonal, if $\left|r_{i}\cap c_{t}\right|\leq1$,
$\forall i\in\mathcal{N}$ and $\forall t\in\mathrm{\mathbb{T}}$.
For a matrix, clearly the rows and columns constitute such orthogonal
partitions. This is the only condition required by~\cite[Theorem 2.7]{Hershkowitz1997}.
Then by a direct calculation of the summation on the right hand side
of~\cite[(2.8)]{Hershkowitz1997}, we can obtain~(\ref{eq:FullHetero}).
This completes the proof.
\end{proof}
\begin{rem}
In general, there are $2^{mN+1}$ inequalities in~(\ref{eq:FullHetero}),
which will be intractable if $N$ is several thousand. When the PEVs
are fully homogeneous, i.e., they share the same set of parameters
$\Omega$, the above result reduces to the well-known majorization
condition \cite{Hao2014,MarshallOlkinArnold2010}, which consists
of only $N+m$ inequalities. In the linear algebra literature, studies
on the matrix feasibility problem are also focused on finding tractable
conditions under limited heterogeneities in the parameters. Adapted
to the PEV charging scenario, if the arrival time $a$ and the charging
rate $p$ are homogeneous, $\underline{E}^{i}=\bar{E}^{i}$, and under
certain monotonicity condition on $u$, the number of inequalities
in~(\ref{eq:FullHetero}) can be significantly reduced \cite{Brualdi2003,Chen2015}.
The author in~\cite{Chen1992} obtained a simple majorization type
result under a special monotonicity condition on $u$. Under this
condition, the charging rate $p$ can be relaxed to be heterogeneous
both among different PEVs and at different time instances. Note that
Theorem~\ref{thm:iff} also applies to the case where the charging
rate takes integer values (see~\cite[Remark (2.19)]{Hershkowitz1997}).
\end{rem}
Since the condition~(\ref{eq:FullHetero}) is very difficult to check
in practice, the goal of this paper is to find a sufficient approximation
of it using much fewer inequalities that does not depend on $N$.
The direct approximation from~(\ref{eq:FullHetero}) could be difficult.
However, it is possible to start with the definition~(\ref{eq:P_AF}).
An interesting perspective is to view the Minkowski sum as a projection
operation. Clearly, from~(\ref{eq:P_AF}), we see that $\mathcal{P}$
is the projection of a higher dimensional polytope $\tilde{\mathcal{P}}$
onto the $u$ subspace, i.e. 
\begin{equation}
\mathcal{P}=\mbox{Proj}_{u}(\tilde{\mathcal{P}}),\label{eq:prj}
\end{equation}
where 
\begin{align}
\tilde{\mathcal{P}}:= & \left\{ (u,\bar{u}):u=\sum_{i=1}^{N}u^{i},~u^{i}\in\mathcal{P}^{i}\right\} ,\label{eq:Ptilde}
\end{align}
$\bar{u}:=(u^{1},u^{2},\cdots,u^{N}),$ and $\mbox{Proj}_{u}$ is
the projection onto the $u$ subspace. In fact, the Minkowski sum
of two sets is often calculated via projection. Note that the number
of the facets of the polytope $\tilde{\mathcal{P}}$ is of $O(mN)$,
as compared to $O(2^{mN})$ of its projection. The relation (\ref{eq:prj})
inspires us to approximate $\mathcal{P}$ based on only the expression
of $\tilde{\mathcal{P}}$. The specific approximation method will
be described in the next section.

\section{Sufficient Approximation of the Aggregate Flexibility}

We will first present our method in a general setting of computing
the maximum homothet of a polytope included in a polytopic projection,
and then apply it to the PEV charging scenario. Our approximation
method is inspired by~\cite{ZhenHertog2015}, where the ellipsoidal
approximation of a polytopic projection is addressed resorting to
the robust optimization technique~\cite{Ben-TalGhaouiNemirovski2009}. 

\subsection{Approximation of the Polytopic Projection }

Given full dimensional polytopes 
\begin{align*}
\mathbb{B}:= & \{u\subset\mathbb{R}^{m}:Fu\leq H\},\\
\tilde{\mathcal{P}}:= & \{(u,\bar{u})\subset\mathbb{R}^{m+\bar{m}}:B\left[\begin{array}{c}
u\\
\bar{u}
\end{array}\right]\leq c\},
\end{align*}
we want to find its maximum homothet of $\mathbb{B}$ contained in
the projection of $\tilde{\mathcal{P}}$ onto the $u$ subspace. It
can be formulated as the following optimization problem
\begin{equation}
\begin{array}{ll}
\underset{\lambda>0,\mu}{\mbox{maximize}} & \lambda\\
\mbox{subject to:} & \lambda\mathbb{B}+\mu\subset\mbox{Proj}_{u}(\tilde{\mathcal{P}}).
\end{array}\label{eq:OPP}
\end{equation}

To facilitate the later formulation of the optimization problems as
linear programs, we perform a change of variables $s=1/\lambda$ and
$r=-\mu/\lambda$. Thus Problem (\ref{eq:OPP}) is equivalent to finding
the minimum homothet of $\mbox{Proj}_{u}(\tilde{\mathcal{P}})$ that
contains $\mathbb{B}$, i.e.
\begin{equation}
\begin{array}{ll}
\underset{s>0,r}{\mbox{minimize}} & s\\
\mbox{subject to:} & \mathbb{B}\subset s\mbox{Proj}_{u}(\tilde{\mathcal{P}})+r.
\end{array}\label{eq:OPP1}
\end{equation}

Furthermore, since orthogonal projection is a linear operation, we
have 
\[
\mbox{Proj}_{u}(s\tilde{\mathcal{P}}+\tilde{r})=s\mbox{Proj}_{u}(\tilde{\mathcal{P}})+r,
\]
where $\tilde{r}=L(r;\mathbf{0}):=\left[\begin{array}{c}
r\\
\mathbf{0}
\end{array}\right]$ is the lift of the vector $r$ in $\mathbb{R}^{m+\bar{m}}$ by setting
the additional dimensions to $\mathbf{0}$. Hence, $L(\mathbb{B};\bar{u}_{0})\subset s\tilde{\mathcal{P}}+\tilde{r}$
for some $\bar{u}_{0}\in\mathbb{R}^{\bar{m}}$ implies the constraint
in (\ref{eq:OPP1}). Therefore, it is sufficient to pose a more restrictive
constraint to obtain a suboptimal solution, and we have
\begin{equation}
\begin{array}{ll}
\underset{s>0,r,\bar{u}_{0}}{\mbox{minimize}} & s\\
\mbox{subject to:} & L(\mathbb{B};\bar{u}_{0})\subset s\tilde{\mathcal{P}}+\tilde{r},
\end{array}\label{eq:OPP2}
\end{equation}
 where the constraint can be expressed as 
\[
\mathbb{B}\subset\left\{ B\left[\begin{array}{c}
u\\
\bar{u}_{0}
\end{array}\right]\leq sc+B\tilde{r}\right\} .
\]

By applying the Farkas's Lemma (see the Appendix \ref{subsec:Farkaslemma}),
the above optimization problem can be transformed into the following
linear programming problem,
\begin{equation}
\begin{array}{cl}
\underset{s>0,G\geq0,r,\bar{u}_{0}}{\mbox{minimize}} & s\\
\mbox{subject to:} & GF=\left[\begin{array}{c}
B_{11}\\
B_{21}
\end{array}\right],\\
 & GH\leq B\left[\begin{array}{c}
r\\
-\bar{u}_{0}
\end{array}\right]+sc.
\end{array}\label{eq:OPP3}
\end{equation}

Before proceed, we illustrate the formulation (\ref{eq:OPP2}) using
a simple example borrowed from~\cite{ZhenHertog2015}. 
\begin{example}
\label{exa:test}Let $\tilde{\mathcal{P}}$ be given by $\tilde{\mathcal{P}}=\{(x,y)|-0.5x-y\leq-9,0.6x+y\leq10,-x-y\leq10\}$,
and $\mathbb{B}=\{x|-0.5\leq x\leq1\}$. The polytope $\tilde{\mathcal{P}}$
is plotted in Fig. \ref{fig:example}. We solve the problem (\ref{eq:OPP3})
to find a sufficient approximation of $\mbox{Proj}_{x}(\tilde{\mathcal{P}})$,
and obtain $s=1.125$, $r=-2.75$, $\bar{u}_{0}=9$. The corresponding
scale factor is $\lambda=1/s=8/9$, and translate factor is $\mu=22/9$.
From these data we can have that $L(\lambda\mathbb{B};8)+[\mu,0]^{T}\subset\tilde{\mathcal{P}}$.
This corresponds to the fact that the longest horizontal line segment
that is contained in $\tilde{\mathcal{P}}$ is at $y=8$.
\end{example}
\begin{center}
\begin{figure}
\begin{centering}
\includegraphics[clip,width=0.9\linewidth]{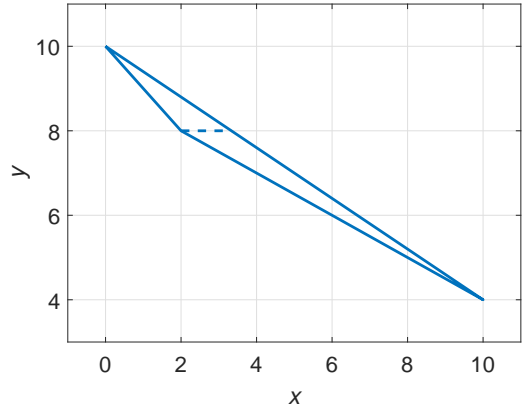}
\par\end{centering}
\caption{\label{fig:example}The approximation of the projection of $\tilde{\mathcal{P}}$
via Problem (\ref{eq:OPP2}).}
\end{figure}
\par\end{center}

Clearly, the formulation of Problem~(\ref{eq:OPP2}) is very conservative.
It actually requires the homothet of the polytope $\mathbb{B}$ be
entirely contained in $\tilde{\mathcal{P}}$. It amounts to each time
fixing $\bar{u}=\bar{u}_{0}$, and then measuring the cross section
of $\tilde{\mathcal{P}}$. However, for approximating the projection
of $\tilde{\mathcal{P}}$, we only need 
\begin{equation}
\forall u\in\lambda\mathbb{B}+\mu,\ \exists\bar{u}(u),\mbox{ such that }(u,\bar{u}(u))\in\tilde{\mathcal{P}},\label{eq:Dependence}
\end{equation}
where $\bar{u}$ is a function of $u$, while in Problem (\ref{eq:OPP2})
$u$ is determined by $\bar{u}$. The relation (\ref{eq:Dependence})
can be interpreted in the context of the adjustable robust optimization
problem~\cite{Ben-TalGoryashkoGuslitzerEtAl2003}, where $u$ is
the so called non-adjustable variable, and $\bar{u}$ is the adjustable
variable. The function $\bar{u}(u)$ is called the decision rule.
Solving~(\ref{eq:OPP2}) over all possible choices of $\bar{u}(u)$
is intractable. An efficient way to overcome this is to restrict the
choice of $\bar{u}(u)$ to be the affine decision rules,
\begin{equation}
\bar{u}=Wu+V,\label{eq:Decision}
\end{equation}
where $W\in\mathbb{R}^{\tilde{m}\times m}$, and $V\in\mathbb{R}^{\tilde{m}}$.
Using (\ref{eq:Decision}) in Problem (\ref{eq:OPP2}) and by some
manipulation, we can obtain
\begin{equation}
\begin{array}{ll}
\underset{s>0,r,W,V}{\mbox{minimize}} & s\\
\mbox{subject to:} & \mathbb{B}\subset\left\{ B\left[\begin{array}{c}
I\\
W
\end{array}\right]u\leq B\left[\begin{array}{c}
r\\
-V
\end{array}\right]+sc\right\} .
\end{array}\label{eq:OPPAPP}
\end{equation}

Using the Farkas's Lemma, Problem (\ref{eq:OPPAPP}) can be solved
by the following linear programming

\[
(\mbox{APP})\qquad\begin{array}{ll}
\underset{s>0,G\geq0,r,W,V}{\mbox{minimize}} & s\\
\mbox{subject to:} & GF=B\left[\begin{array}{c}
I\\
W
\end{array}\right],\\
 & GH\leq B\left[\begin{array}{c}
r\\
-V
\end{array}\right]+sc.
\end{array}
\]

We test the above formulation by computing the approximation of the
polytopic projection in Example \ref{exa:test}.
\begin{example}[Continued]
 We solve the Problem (APP) and obtain $r=-0.5$, $s=0.15$. This
corresponds to the interval $1/0.15\left(\mathbb{B}+0.5\right)=[0,10]$,
which is exactly $\mbox{Proj}_{x}(\tilde{\mathcal{P}})$.
\end{example}
In general, the Problem (APP) gives a suboptimal solution for the
approximation of $\mbox{Proj}_{x}(\tilde{\mathcal{P}})$ with respective
to $\mathbb{B}$. A possible way to reduce the conservativeness is
to employ the quadratic decision rule or other nonlinear decision
rules as reported in~\cite{Ben-TalGoryashkoGuslitzerEtAl2003}.

\subsection{Aggregation of the PEVs' Flexibility}

In this subsection, the polytopic projection approximation developed
in the above section will be employed to aggregate the PEVs' flexibility.
We will discuss several issues including the choice of the nominal
model $\mathbb{B}$, the preprocessing of the charging constraints,
and the strategy for parallel computation. Finally, the explicit formulae
for the flexibility model and the corresponding scheduling policy
are derived.

\subsubsection{Choice of the Nominal Model}

Intuitively, one can choose the nominal polytope $\mathbb{B}$ to
be of the similar form of~(\ref{eq:pi}), and the parameters can
be taken as the mean values of the PEV group. More generally, we can
define the virtual battery model as follows.
\begin{defn}
The set $\mathbb{B}(\phi)$ is called a $m$-horizon discrete time
virtual battery with parameters $\phi:=\{\underline{p},\bar{p},\underline{E},\bar{E}\}$,
if 
\[
\mathbb{B}(\phi):=\{u\in\mathbb{R}^{m}:\underline{p}\leq u\leq\bar{p}\mbox{ and }{\bf 1}_{m}^{T}u\in[\underline{E},\bar{E}]\}.
\]
$\mathbb{B}(\phi)$ is called a sufficient battery if $\mathbb{B}(\phi)\subset\mathcal{P}$.
\end{defn}
Conceptually, the virtual battery model mimics the charging/discharging
dynamics of a battery. We can regard $u$ as the power draw of the
battery,  $\underline{p}$ and $\bar{p}$ as its discharging/charging
power limits, and $\bar{E}$ and $\underline{E}$ as the energy capacity
limits. Geometrically, it is a polytope in $\mathbb{R}^{m}$ with
$2m+2$ facets, which is computationally very efficient when posed
as the constraint in various optimization problems. 

\subsubsection{Preprocessing the Charging Constraints}

Note that the original high-dimensional polytope $\tilde{\mathcal{P}}$
defined in~(\ref{eq:Ptilde}) contains equality constraints, which
is not full dimensional. Therefore, first we have to remove the equalities
by substituting the variables. For simplicity, assume that $\bar{E}^{i}>\underline{E}^{i}$.
More explicitly, let $u_{t}^{i}$ be the $i^{\mbox{th}}$ PEV's charging
profile at time $t$, $\forall i\in\mathcal{N}$, and $\forall t\in\mathbb{T}$,
$u_{t}$ be the generation profile at time $t$. The overall charging
constraints of the PEV group can be written as follows
\begin{equation}
\begin{cases}
u_{t}=\sum_{i=1}^{N}\mathbbm{1}_{\mathbb{A}^{i}}(t)u_{t}^{i}, & \forall t\in\mathbb{T},\\
0\leq u_{t}^{i}\leq p^{i}, & \forall i\in\mathcal{N},\mbox{ and }\forall t\in\mathbb{A}^{i},\\
u_{t}^{i}=0, & \forall i\in\mathcal{N},\mbox{ and }\forall t\in\mathbb{T}\backslash\mathbb{A}^{i},\\
\underline{E}^{i}\leq\sum_{t=a^{i}}^{d^{i}}u_{t}^{i}\leq\bar{E}^{i}, & \forall i\in\mathcal{N},
\end{cases}\label{eq:original}
\end{equation}
which is a polytope in $\mathbb{R}^{(N+1)m}$, and the coordinate
is designated to be $(u,u^{1},u^{2},\cdots,u^{N})$. We first need
to eliminate the equality constraints containing $(u^{1},u^{2},\cdots,u^{N})$,
i.e., the first line of~(\ref{eq:original}). To standardize the
elimination process, define $N_{t}:=\{i\in\mathcal{N}|t\in\mathbb{A}^{i}\}$,
which is the index set of the PEVs that can be charged at time $t$.
Without loss of generality, assume that $N_{t}\neq\emptyset$, $\forall t\in\mathbb{T}$,
and we substitute $u_{t}^{j_{t}}=u_{t}-\sum_{i\in N_{t}\backslash j_{t}}u_{t}^{i}$,
$\forall t\in\mathbb{T}$ in~(\ref{eq:original}), where $j_{t}:=\min_{i\in N_{t}}i$
is the first PEV in the set $N_{t}$. Let $S_{i}:=\{t\in\mathbb{A}^{i}:j_{t}=i\}$
be the set of time instants at which the substitution of $u_{t}^{i}$
is made. Further, we remove the coordinate $u_{t}^{i}=0$ and obtain,
\begin{equation}
\begin{cases}
0\leq u_{t}-\sum_{i\in N_{t}\backslash\{j_{t}\}}u_{t}^{i}\leq p^{j_{t}}, & \forall t\in\mathbb{T},\\
0\leq u_{t}^{i}\leq p^{i}, & \forall t\in\mathbb{T},i\in N_{t}\backslash\{j_{t}\},\\
\underline{E}^{i}\leq\sum_{t\in\mathbb{A}^{i}\backslash S_{i}}u_{t}^{i}\\
+\sum_{t\in S_{i}}\left\{ u_{t}-\sum_{i\in N_{t}\backslash j_{t}}u_{t}^{i}\right\} \leq\bar{E}^{i}, & \forall i\in\mathcal{N},
\end{cases}\label{eq:eliminated}
\end{equation}
where the new coordinate becomes $(u,\tilde{u}),$ with 
\begin{equation}
\tilde{u}:=(\tilde{u}^{1},\tilde{u}^{2},\cdots,\tilde{u}^{N})\label{eq:lowerDimofP}
\end{equation}
and $\tilde{u}^{i}=[u_{t}^{i}]_{t\in\mathbb{A}^{i}\backslash S_{i}}$,
$i\in\mathcal{N}$. Clearly, $\tilde{u}$ has a dimension $\tilde{m}:=\sum_{i=1}^{N}(d^{i}-a^{i}+1)-m$,
and there are a number of $n:=2(N+\sum_{i=1}^{N}(d^{i}-a^{i}+1))$
linear inequalities in (\ref{eq:eliminated}).  We denote it by the
matrix form, 
\begin{equation}
B\left[\begin{array}{c}
u\\
\tilde{u}
\end{array}\right]\leq c,\label{eq:Bm}
\end{equation}
where note that $B\in\mathbb{R}^{n\times(m+\tilde{m})}$ is a sparse
matrix and has the structure, 
\[
B=\left[\begin{array}{cc}
B_{11} & B_{12}\\
0_{2\tilde{m}\times m} & B_{22}
\end{array}\right].
\]

\subsubsection{Scalability}

For a fixed time horizon $m$, both the number of the decision variables
and the number of inequality constraints of Problem~(APP) increase
linearly with $\tilde{m}$. When the number of PEVs to be aggregated
is too large, solving Problem~(APP) would be intractable. To address
the increasing numerical complexity, we propose to divide the PEVs
into small groups, and solve Problem~(APP) for each group with respect
to the same nominal model $\mathbb{B}$. Denoting the solutions of
Problem~(APP) for the $k^{\mbox{th}}$ group by $(s_{k},r_{k},W_{k},V_{k},G_{k})$,
the aggregate flexibility of the $k^{\mbox{th}}$ group is given by
$1/s_{k}(\mathbb{B}-r_{k})$. Then the flexibility of the overall
PEV group can be calculated directly based on the following lemma.
\begin{lem}
\label{lem:MkskSumScale}Let $\mathbb{B}$ be a polytope, and $\lambda_{1}$,
$\lambda_{2}$ be non-negative scalars. Then $\lambda_{1}\mathbb{B}\oplus\lambda_{2}\mathbb{B}=(\lambda_{1}+\lambda_{2})\mathbb{B}$.
\end{lem}
The above result can be easily verified. A more general proof for
the convex body can be found in~\cite[Remark 1.1.1]{Schneider1993}).

By this ``divide and conquer'' strategy, the original highly complex
optimization problem can be solved very efficiently in a parallel
fashion. However, this increases the conservativeness of the approximation,
which is a result of the trade-off between the tractability and the
optimality.

In case that different nominal models $\mathbb{B}^{k}$ are used for
each group, we can perform the aggregation again over the obtained
groups. Repeating this after several stages, we can arrive one virtual
battery model for the overall PEV group. Even though we have to spread
the computation over time in different stages, in practice this process
terminates soon since the number of stages is of order $\log_{x}^{N}$
when $x$ PEVs/groups are processed at each run of Problem~(APP).

\subsubsection{Design of the Sufficient Virtual Battery}

Combining the above development, the following explicit formulae for
designing the sufficient virtual battery can be derived. The scheduling
policy for each individual PEVs can also be obtained. Without loss
of generality, these results are stated for the case where only one
stage aggregation is executed. The formulae for multi-stage aggregation
can be obtained analogously. For convenience, let us denote the solutions
of the Problem (APP) by the output of the function $(s_{k},r_{k},W_{k},V_{k},G_{k})=\mbox{APP}(\mathcal{\tilde{P}}^{k},\mathbb{B}),$
where $\mathcal{\tilde{P}}^{k}$ is the high-dimensional polytope
associated with the $k^{\mbox{th}}$ group of PEVs, and $\mathbb{B}$
is a given nominal model parameterized by $(\bar{p},\underline{p},\bar{E},\underline{E})$.
The proof of the following theorem can be found in the Appendix \ref{subsec:Proof-of-Theorem}. 
\begin{thm}
\label{thm:SuffVB}Suppose $(s_{k},r_{k},W_{k},V_{k},G_{k})=\mbox{APP}(\tilde{\mathcal{P}}^{k},\mathbb{B})$,
$\forall k$. Then $\mathbb{B}_{s}=\lambda\mathbb{B}+\mu$ is a sufficient
battery parameterized by\begin{subequations}
\begin{alignat*}{1}
\bar{u}_{s}= & \lambda\bar{p}+\mu,\\
\underline{u}_{s}= & \lambda\underline{p}+\mu,\\
\bar{E}_{s}= & \lambda\bar{E}+{\bf 1}_{m}^{T}\mu,\\
\underline{E}_{s}= & \lambda\underline{E}+{\bf 1}_{m}^{T}\mu,
\end{alignat*}
\end{subequations}where $\lambda=\sum_{k}\lambda_{k}$, $\mu=\sum_{k}\mu_{k}$
and $\lambda_{k}=\frac{1}{s_{k}}$ , $\mu_{k}=\frac{-r_{k}}{s_{k}}$.
Furthermore, $\forall u\in\mathbb{B}_{s},$ the scheduling policy
is given by 
\begin{equation}
\tilde{u}_{k}=W_{k}\frac{\lambda_{k}}{\lambda}(u-\mu)+\lambda_{k}V_{k},\label{eq:SchdPlcy}
\end{equation}
where $\tilde{u}_{k}=(\tilde{u}_{k}^{1},\tilde{u}_{k}^{2},\cdots,\tilde{u}_{k}^{N_{k}})$
denotes the charging profiles of the number of $N_{k}$ PEVs in the
$k^{\mbox{th}}$ group.
\end{thm}
\begin{center}
\begin{figure}[t]
\begin{centering}
\includegraphics[clip,width=0.9\linewidth]{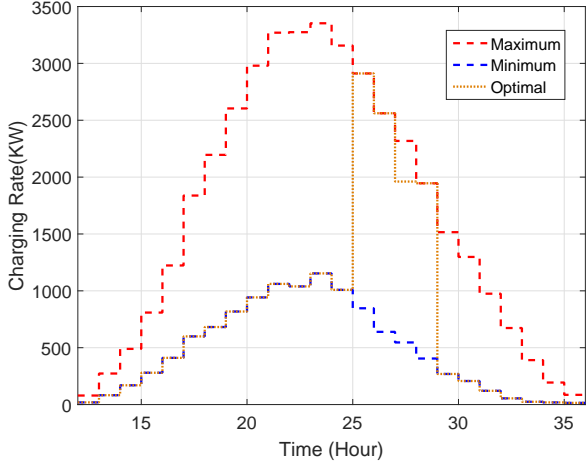}
\par\end{centering}
\caption{\label{fig:bounds}The maximum and the minimum charging rates of the
designed virtual battery model.}
\end{figure}
\par\end{center}

\begin{center}
\begin{figure}[t]
\begin{centering}
\includegraphics[clip,width=1\linewidth]{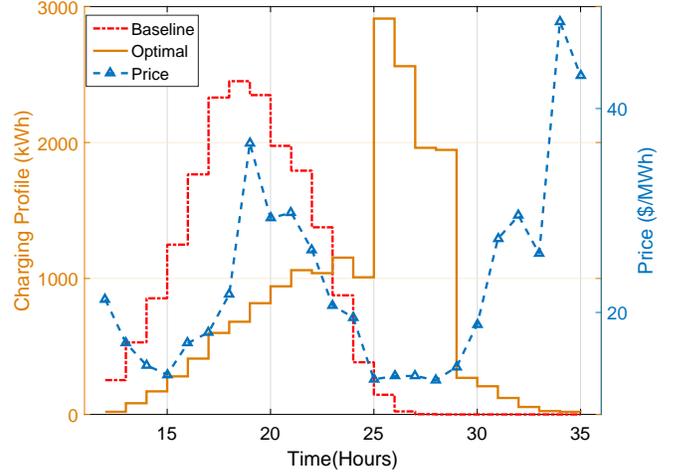}
\par\end{centering}
\caption{\label{fig:planing}Aggregate charging profile versus the price.}
\end{figure}
\par\end{center}

\section{Simulation}

In this section, we consider coordination of a group of 1000 PEVs
for energy arbitrage. The considered time horizon is 24 hours, and
the price is taken as the Day-Ahead Energy Market locational marginal
pricing (LMP)~\cite{PJM}. The parameters of PEVs are randomly generated
by their types and the corresponding probability distributions (see~\cite{GuoZhangYanEtAl2013}
for more details). Since most of the PEVs arrive during the afternoon
to midnight and leave during the next 12 hours, we choose to simulate
from 12:00 noon to the same time on the next day. In addition, we
assume a $\pm5\%$ total charging energy flexibility around the nominal
energy requirement. The (APP) problem is solved using the GLPK linear
programming solver~\cite{glpk2006} interfaced with YALMIP~\cite{Lofberg2004}.
At the first stage, we randomly divide the 1000 PEVs into 100 groups,
where each group contains 10 PEVs. The aggregate flexibility approximation
is thus solved for 10 PEVs in one group. This number is chosen according
to the numerical efficiency of the solver. The parameters of the nominal
battery model $\mathbb{B}$ for each group are chosen as the average
values of the group. For those groups having the same minimum arrival
time and maximum departure time, we take the average values between
them and approximate their flexibilities using the same nominal model.
For example, by setting the elements in the charging matrix $M$ by
their maximum values, the upper charging limits $\bar{p}$ for the
nominal battery model are calculated as the column averages of $M$.
Since these groups share the same nominal model, their approximated
aggregated flexibilities can be calculated very easily based on Lemma~\ref{lem:MkskSumScale}.
In our simulation, after the first stage, the 100 groups are merged
into 22 collections which are represented by polytopes of different
codimensions. Then we repeat the process in the first stage to approximate
the flexibility of these 22 collections, where each time we aggregate
11 collections. Finally a sufficient battery model $\mathbb{B}_{s}$
is obtained for the characterization of the aggregate flexibility
of the entire loads group.

The dynamic charging limits of the obtained battery model are illustrated
in Fig.~\ref{fig:bounds}. The total charging energy bounds are obtained
as $\underline{E}=18.35$ MWh and $\bar{E}=19.09$ MWh, which lies
in the true aggregate energy consumption interval $l_{E}:=[17.65,19.51]$
MWh. From Fig.~\ref{fig:bounds}, we can see that around the midnight
(the $24^{\mbox{th}}$ hour), the charging flexibility of the PEVs
are the largest in terms of the difference of the charging rate bounds.
Denoting the energy price by $\pi$, and the planned energy by $z$,
the energy arbitrage problem can be formulated as a linear programming
problem as follows,
\[
\begin{array}{ll}
\underset{z}{\mbox{minimize}} & \pi^{T}z\\
\mbox{subject to:} & z\in\mathbb{B}_{s}.
\end{array}
\]

Clearly, the above optimization problem can be solved much more efficiently
than directly optimizing the power profiles subject to the constraints
of 1000 PEVs. We plot the obtained power profiles against the price
changes in Fig.~\ref{fig:planing}. It can be observed that most
of the energy demand are consumed during $1:00\mbox{AM}$~$(25^{\mbox{th}}\mbox{ hour})$
to $4\mbox{AM}$~$(28^{\mbox{th}}\mbox{ hour})$ in the morning,
when the prices are at its lowest. The same curve of the planned power
is also plotted in Fig.~\ref{fig:bounds} (the dotted line), where
note we assume that the time discretization unit is $1$ hour. We
see that the planned charging rate lies in the charging bounds of
$\mathbb{B}_{s}$, and almost always matches the maximum/minimum bound.
Using this charging profile, the total energy being charged to the
PEVs is $18.34$ MWh which lies in the interval $l_{E}$. Hence, it
is adequate and the charging requirement of individual PEV can be
guaranteed by using the scheduling policy (\ref{eq:SchdPlcy}).

We choose the immediate charging policy as the baseline and use it
to compare with the obtained optimal charging profile in Fig.~\ref{fig:bounds}.
To ensure a fair comparison, we impose an additional constraint that
the total energies consumed by both profiles are the same. The total
energy cost for the baseline charging profile is $430.8\$$, while
the cost for the optimal charging profile is $343.0\$$, which reduces
the baseline cost by about $20\%$.

\section{Conclusions and Future Work}

This paper proposed a novel polytopic projection approximation method
for extracting the aggregate flexibility of a group of heterogeneous
deferrable loads. The aggregate flexibility of the entire load group
could be extracted parallelly and in multiple stages by solving a
number of linear programming problems. The scheduling policy for individual
load was simultaneously derived from the aggregation process. Finally,
a PEV energy arbitrage problem was solved to demonstrate the effectiveness
of our approach at characterizing the feasible aggregate power flexibility
set, and facilitating finding the optimal power profile.

Our future work includes studying the performances of using other
decision rules such as the quadratic decision rule and the nonlinear
decision rule, and as compared to the method using Zonotopes. In addition,
it is interesting to consider a probabilistic description of the aggregate
flexibility as in practice the uncertainty of the loads parameters
must be considered to reduce the risk of over-estimating or under-estimating
the aggregated flexibility.

\appendix{}

\subsection{Farkas' lemma\label{subsec:Farkaslemma}}

For the sake of completeness, we restate the following version of
Farkas's lemma as used in~\cite{Eaves1982}, which will be used to
derive our algorithm for approximating the polytopic projection. Its
proof can be found in~\cite{Mangasarian2002}.
\begin{lem}
\label{lem:Farkas}(Farkas' lemma) Suppose that the system of inequalities
$Ax\leq b,$ $A\in\mathbb{R}^{m\times n}$ has a solution and that
every solution satisfies $Fx\leq d,$ $F\in\mathbb{R}^{k\times n}$.
Then there exists $G\in\mathbb{R}^{k\times m}$ , $G\geq0$, such
that $GA=F$ and $Gb\leq d$.
\end{lem}

\subsection{\label{subsec:Proof-of-Theorem}Proof of Theorem \ref{thm:SuffVB}}

(1) Sufficient battery: since $(s_{k},r_{k})$ is the solution of
the APP problem, we have $\lambda_{k}\mathbb{B}+\mu_{k}\subset\tilde{\mathcal{P}}^{k}$,
where $\lambda_{k}=1/s_{k}$ and $\mu_{k}=-r_{k}/s_{k}$. Let 
\[
\mathbb{B}_{s}:=\biguplus_{k}(\lambda_{k}\mathbb{B}+\mu_{k}),
\]
and it can be shown that, 
\begin{align*}
\mathbb{B}_{s} & =\biguplus_{k}\lambda_{k}\mathbb{B}+\mu\\
 & =\lambda\mathbb{B}+\mu,
\end{align*}
where $\lambda=\sum_{k}\lambda_{k},\mu=\sum_{k}\mu_{k}$, and the
last equality is derived by using Lemma~\ref{lem:MkskSumScale}.
Since 
\[
\biguplus_{k}(\lambda_{k}\mathbb{B}+\mu_{k})\subset\biguplus_{k}\tilde{\mathcal{P}}^{k}=\tilde{\mathcal{P}},
\]
we see $\mathbb{B}_{s}$ is sufficient. Now suppose $u\in\mathbb{B}_{s}$,
then $(u-\mu)/\lambda\in\mathbb{B}$. Denoting $\mathbb{B}$ by its
facet representation, we have $\mathbb{B}=\{Fu\leq H\}$, where
\begin{align*}
F= & [I_{m},I_{m},{\bf 1}_{m},{\bf 1}_{m}]^{T},\\
H= & (\bar{u},-\underline{u},\bar{E},-\underline{E}),
\end{align*}
and then parameters of $\mathbb{B}_{s}$ can be obtained from 
\[
Fu\leq\lambda H+F\mu.
\]

(2) Scheduling policy: given a generation profile $u$ in $\mathbb{B}_{s}$,
we can decompose it into the individual admissible power profile $\tilde{u}$
through two steps. First, we decompose it into the generation profiles
for each groups: $\forall u\in\mathbb{B}_{s}$, by part (1) we know
\[
\frac{1}{\lambda}(u-\mu)\in\mathbb{B},
\]
and further more 
\[
\frac{\lambda_{k}}{\lambda}(u-\mu)+\mu_{k}\in\tilde{\mathcal{P}}^{k}.
\]

Denoting the generation profile for the $k^{\mbox{th}}$ group by
\[
z_{k}:=\frac{\lambda_{k}}{\lambda}(u-\mu)+\mu_{k},
\]
and hence, $z_{k}\in\lambda_{k}\mathbb{B}+\mu_{k}$ which is the aggregate
flexibility extracted from the $k^{\mbox{th}}$ group. It can be further
decomposed into each PEVs in the $k^{\mbox{th}}$ group. Now we need
to use the linear decision rule in (\ref{eq:Decision}). Note that
the decision rule (\ref{eq:Decision}) is applied in (\ref{eq:OPPAPP})
which actually maps from $\mathbb{B}$ to $s_{k}\tilde{\mathcal{P}}^{k}+(r_{k},\mathbf{0})$,
while the decomposition mapping we need is actually from $\lambda_{k}\mathbb{B}+\mu_{k}$
to $\tilde{\mathcal{P}}^{k}$. The mappings between these four polytopes
form a commutative diagram (see below). Observe that the linear ratio
of the decomposition mapping does not change, and only the translate
vector needs to be calculated. Therefore, assume that the decomposition
takes the form 
\[
\tilde{u}_{k}=W_{k}z_{k}+U,
\]
where $U$ is the translate vector to be determined.

\begin{center}
\includegraphics[clip,width=0.6\linewidth]{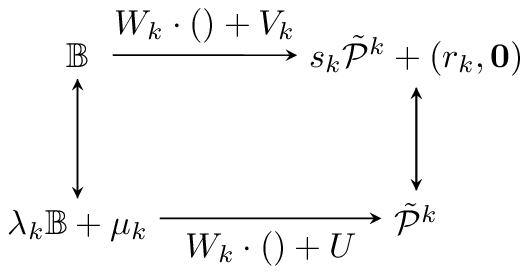}
\par\end{center}

From the above commutative diagram, we must have $\forall z_{k}\in\lambda_{k}\mathbb{B}+\mu_{k}$,
\[
s_{k}\left(W_{k}z_{k}+U\right)+\mathbf{0}=W_{k}\left(s_{k}z_{k}+r_{k}\right)+V_{k}.
\]
Solve the above equality and we will have 
\[
U=\lambda_{k}V_{k}-W_{k}\mu_{k},
\]
and the overall scheduling policy (\ref{eq:SchdPlcy}) follows from
the composition of $z_{k}$ and $\tilde{u}_{k}$.\hfill \qed

\bibliographystyle{IEEEtranS}

\end{document}